%% file: pi0_functoriality_arxiv2.tex
\documentclass[12pt]{amsart}
\usepackage{amsmath,amscd,amsthm,amsfonts, amssymb,amsxtra}
\usepackage{calrsfs}
\usepackage{lscape}
\usepackage[all]{xy} \SelectTips{cm}{}
\usepackage{enumitem}
\usepackage[dvipsnames]{xcolor}
  \usepackage{graphicx}
  \usepackage[margin=1.75in]{geometry}
  \usepackage[pagewise, displaymath, mathlines]{lineno}
  
\CDat
\include{preamble_refs}

\newtheorem*{convent}{Conventions}
\newtheorem*{acks}{Acknowledgements}

\DeclareMathOperator{\id}{id}
\DeclareMathOperator{\maps}{map}

\DeclareSymbolFont{bbold}{U}{bbold}{m}{n}
\DeclareSymbolFontAlphabet{\mathbbold}{bbold}

\begin{document}
\title{On the multiplicativity of the Euler characteristic}
\date{\today} 
\author{John R. Klein} 
\address{Department of Mathematics, Wayne State University,
Detroit, MI 48202} 
\email{klein@math.wayne.edu} 
\author{Cary Malkiewich} 
\address{Department of Mathematics,
Binghamton University,
Binghamton, NY 13902}
\email{malkiewich@math.binghamton.edu}
\author{Maxime Ramzi} 
\address{Department of Mathematical Sciences,
University of Copenhagen,
Universitetsparken 5,
DK-2100 Copenhagen Ø}
\email{ramzi@math.ku.dk}

\begin{abstract}  
	We give two proofs that the Euler characteristic is multiplicative, for fiber sequences of finitely dominated spaces. This is equivalent to proving that the Becker-Gottlieb transfer is functorial on $\pi_0$.
\end{abstract}

\maketitle
\setlength{\parindent}{15pt}
\setlength{\parskip}{1pt plus 0pt minus 1pt}
\def\Top{\bold T\bold o \bold p}
\def\wTop{\text{\rm w}\bold T}
\def\wT{\text{\rm w}\bold T}
\def\vo{\varOmega}
\def\vs{\varSigma}
\def\smsh{\wedge}
\def\esmsh{\,\hat\wedge\,}
\def\flush{\flushpar}
\def\dbslash{\smsh!\! /}
\def\:{\colon\!}
\def\Bbb{\mathbb}
\def\bold{\mathbf}
\def\cal{\mathcal}
\def\orb{\cal O}
\def\hoP{\text{\rm ho}P}
\def\Sph{\mathbb{S}}
\def\Z{\mathbb{Z}}
\def\Q{\mathbb{Q}}
\def\F{\mathbb{F}}
\def\sma{\wedge}
\def\ti{\tilde}
\newcommand{\HH}{\mathrm{HH}}
\newcommand{\THH}{\mathrm{THH}}

\setcounter{tocdepth}{1}
\tableofcontents
\addcontentsline{file}{sec_unit}{entry}

\section{Introduction}

The Euler characteristic $\chi(X) \in \Z$ is among the oldest and simplest invariants of topological spaces. It is defined when $X$ is homotopy equivalent to a finite cell complex, as the alternating sum of the numbers of cells of $X$, or equivalently as the alternating sum of the ranks of the homology groups $H_i(X)$.

It is well-known that the Euler characteristic is multiplicative in the sense that $\chi(X \times Y) = \chi(X) \chi(Y)$
for finite cell complexes $X$ and $Y$. More generally, given a fibration sequence
\[ F \to E \to B \]
in which $F$ and $B$ are finite cell complexes, $E$ is also equivalent to a finite cell complex and we have that $\chi(E) = \chi(F)\chi(B)$.

In this paper, we prove that this multiplicativity holds in the more general case that $F$ and $B$ are finitely dominated. Recall that a \emph{finitely dominated} space is one that is a retract up to homotopy of a finite cell complex. If $B$ is finitely dominated, its integral Euler characteristic $\chi(B) \in \Z$ can be defined in several equivalent ways, for instance:
\begin{itemize}
	\item the alternating sum of the ranks of the homology groups, with coefficients in any field $H_i(B;k)$,
	\item the Euler characteristic of any finite complex homotopy equivalent to $\Sigma^2 B$,
	\item the class in $K_0(\Z) \cong \Z$ defined by the perfect chain complex $C_*(B)$, etc.
\end{itemize}

\begin{bigthm}\label{intro_multiplicative}
For every fiber sequence of spaces
$$
F\to E \to B
$$
where $F$ and $B$ (and therefore $E$) are finitely dominated,
\[ \chi(E) = \chi(F)\chi(B). \]
\end{bigthm}

\begin{rmk}
	Although the statement of \autoref{intro_multiplicative} looks elementary, its proof is not, and as far as we can tell it does not appear previously in the literature. While in the case of a finite cell complex $B$, one can induct over cells, in the finitely dominated case, this induction procedure breaks. Both of the proofs in this paper rely on nontrivial theorems concerning $K_0$ of group rings.
\end{rmk}

Our interest in this comes from the question of functoriality for the Becker-Gottlieb transfer. Suppose that $p\colon X \to Y$ is a fibration with finitely dominated fibers. The Becker-Gottlieb transfer $p^!$ is a stable map
\[ p^!\colon \Sigma^\infty_+ Y \to \Sigma^\infty_+ X, \]
defined as the trace of the fiberwise diagonal $X \to X \times_Y X$ in parametrized spectra over $X$, pushed forward from $X$ to $*$.

We say the Becker-Gottlieb transfer is functorial if for fibrations $p\colon X \to Y$ and $q\colon Y \to Z$ with finitely dominated fibers, the composite of transfer maps $p^! \circ q^!$ and the transfer of the composite $(q \circ p)^!$ agree in the stable homotopy category. This is known to hold in some cases, for instance for smooth bundles with compact manifold fiber, but is not known in general.
The first two authors offered a general proof in \cite{km}, but a mistake in that proof was recently found by Bastiaan Cnossen and Shachar Carmeli \cite{km_corr}. As a result, the question is open.

To search for possible counterexamples, we work one homotopy group at a time. We say the Becker-Gottlieb transfer is functorial on $\pi_0$ if for any such $p$ and $q$, the maps $p^! \circ q^!$ and $(q \circ p)^!$ agree after applying $\pi_0$. In other words, the diagram
\[ \xymatrix{
	\pi_0(\Sigma^\infty_+ X) & \pi_0(\Sigma^\infty_+ Y) \ar[l]_-{p^!} & \pi_0(\Sigma^\infty_+ Z) \ar[l]_-{q^!} \ar@/_2em/[ll]_-{(q \circ p)^!}
} \]
commutes.
\begin{bigthm}\label{intro_functorial}
The Becker-Gottlieb transfer is functorial on $\pi_0$.
\end{bigthm}

As a result, any obstruction to functoriality will have to appear above the level of $\pi_0$.

Theorems \ref{intro_multiplicative} and \ref{intro_functorial} are closely related. There is a short argument that each one implies the other, see \autoref{equivalence}. Therefore, in giving two unrelated proofs of Theorems \ref{intro_multiplicative} and \ref{intro_functorial}, we are in fact giving two proofs for each. Our first proof uses a result of Swan on $K_0$ of group rings.
Our second proof uses current work of the third author and collaborators \cite{CCRY}, and a result of Linnell on the Hattori-Stallings trace for group rings.

\begin{convent}
All maps of spaces and spectra in this paper are considered as morphisms in the homotopy category. All functors on spaces and spectra, such as the smash product of two spectra, are derived so that the functor is defined on the homotopy category.
\end{convent} 

\begin{acks}
This paper came about because of an error in \cite{km} that was discovered by Shachar Carmeli and Bastiaan Cnossen. The authors would like to thank them for finding this error and bringing it to their attention, and to thank Bastiaan Cnossen, Shachar Carmeli, and Lior Yanovski for their support in the writing of the current paper. The authors would also like to thank the anonymous referee for helpful comments that improved the exposition at several key points.

Preliminary sketches of both proofs of the main result
 first appeared on MathOverflow in early March, 2022 (see \cite{MO-multip-Euler}).
The authors are grateful to Oscar Randal-Williams for an insightful suggestion appearing there that completed the first version of the first proof.

The material in \S3 was supported  by the U.S.~Department of Energy, Office of Science, 
under Award Number DE-SC-SC0022134. The second author was supported the NSF grants DMS-2005524 and DMS-2052923. The third author was supported by the Danish National Research Foundation through the Copenhagen Centre for Geometry and Topology (DNRF151). 
\end{acks}

\section{Preliminaries}

We begin by recalling some properties of finitely dominated spaces and of the Becker-Gottlieb transfer. We use these to show that Theorems \ref{intro_multiplicative} and \ref{intro_functorial} are equivalent.

As in the introduction, a \emph{finitely dominated} space is one that is a retract up to homotopy of a finite cell complex. The following well known result (see e.g. \cite[Thm 1]{lal}) is needed for the statement of \autoref{intro_multiplicative} to make sense.
\begin{thm}\label{extension_of_fd}
	If $F \to E \to B$ is a fiber sequence in which $F$ and $B$ are finitely dominated then so is $E$.
\end{thm}

Suppose $p\colon X \to Y$ is a fibration whose fibers $X_y$ are finitely dominated. Then the fiberwise suspension spectrum $\Sigma^\infty_{+Y} X$ over $Y$ is dualizable in the symmetric monoidal category of parametrized spectra over $Y$, defined in e.g. \cite{crabb_james,ms,malkiewich_parametrized}
. The \emph{pretransfer} is defined to be the trace of its fiberwise diagonal map $\delta\colon X \to X \times_Y X$. Concretely, in the fiber over each point $y \in Y$, this is the composite of coevaluation, diagonal, and evaluation maps:
\begin{equation}\label{pretransfer}
	\mathbb S
	\to \Sigma^\infty_+ X_y \sma (\Sigma^\infty_+ X_y)^\vee
	\to \Sigma^\infty_+ X_y \sma \Sigma^\infty_+X_y \sma (\Sigma^\infty_+X_y)^\vee
	\to \Sigma^\infty_+ X_y.
\end{equation}
This forms a parametrized stable map from $Y$ to $X$. Pushing forward along the collapse map $Y \to *$ gives a stable map
\[ p^!\colon \Sigma^\infty_+ Y \to \Sigma^\infty_+ X \]
called the \emph{Becker-Gottlieb transfer}. See e.g. \cite{becker1976transfer} \cite[15.2.1]{ms}, \cite[7.1.5]{malkiewich_parametrized}, \cite[\S 4]{km}, \cite{LM}, or \cite[\S 6]{CCRY} for more details.

The additivity of traces from e.g. (\cite{may_additivity,ponto_shulman_additivity,CCRY}) implies that the Becker-Gottlieb transfer is additive over components in the following sense. Let $Y = \coprod_\alpha Y_\alpha$ and $X = \coprod_{\alpha,\beta} X_{\alpha\beta}$ be any decomposition so that each $Y_\alpha$ has as its preimage those terms of the form $X_{\alpha\beta}$. Note that for each $\alpha$, there are only finitely many values of $\beta$. Let $p_{\alpha\beta}\colon X_{\alpha\beta} \to Y_\alpha$ be the restriction of $p$ to $X_{\alpha\beta}$. The transfer $p^!$ is then a map of corproducts
\[ \bigvee_\alpha \Sigma^\infty_+ Y_\alpha \to \bigvee_{\alpha,\beta} \Sigma^\infty_+ X_{\alpha\beta}, \]
so it is determined by its restriction to $\Sigma^\infty_+ Y_\alpha$ for each $\alpha$.
\begin{prop}\label{transfer_additive}
	 The restriction of $p^!$ to $\Sigma^\infty_+ Y_\alpha$ equal to the sum over the terms $X_{\alpha\beta}$ of the the transfers $p_{\alpha\beta}^!$.
\end{prop}

The transfer can be computed on $\pi_0$ by restricting to a single fiber, where it becomes the composite \eqref{pretransfer} followed by the inclusion $\Sigma^\infty_+ X_y \to \Sigma^\infty_+ X$. Taking rational homology turns the smash products in \eqref{pretransfer} into tensor products. If $X$ and $Y$ are connected, this simplifies further to the trace of the identity for $H_*(X_y;\Q)$, in other words the Euler characteristic of the fiber. This proves another standard fact:

\begin{prop}\label{transfer_on_pi0}
	For $X$ and $Y$ connected, the transfer $p^!$ on $\pi_0$ is the map $\Z \to \Z$ that multiplies by the Euler characteristic $\chi(X_y)$.
\end{prop}

Now we may show that Theorems \ref{intro_multiplicative} and \ref{intro_functorial} are equivalent.

\begin{prop}\label{equivalence}
	The Becker-Gottlieb transfer is functorial on $\pi_0$ iff the Euler characteristic is multiplicative for all fiber sequences of finitely dominated spaces.
\end{prop}

\begin{proof}
	We first prove ``$\Leftarrow$.'' Let $p\colon X \to Y$ and $q\colon Y \to Z$ be fibrations with finitely dominated fibers. Since the transfer is additive over components of both the base and total space (\autoref{transfer_additive}), without loss of generality $X$, $Y$ and $Z$ are connected. Let $F$, $E$, and $B$ be defined as pullbacks of $X$, $Y$, and $Z$ as follows:
	\[ \xymatrix{
		F \ar[d] \ar[r] & E \ar[d] \ar[r] & X \ar[d]^-p \\
		{*} \ar[r] & B \ar[d] \ar[r] & Y \ar[d]^-q \\
		& {*} \ar[r] & Z.
	} \]
	By \autoref{transfer_on_pi0}, on $\pi_0$, $q^!$ multiplies by $\chi(B)$, $p^!$ multiplies by $\chi(F)$, and $(q \circ p)^!$ multiplies by $\chi(E)$. So if the Euler characteristic is multiplicative, then the transfer is functorial on $\pi_0$.
	
	Now we prove ``$\Rightarrow$.'' Let $F \to E \to B$ be a fiber sequence of finitely dominated spaces. The Euler characteristic is also additive over components, so without loss of generality $B$ and $E$ are connected. Let $p$ be the map $E \to B$ and $q$ the projection $B \to *$. Then all three of their suspension spectra have $\pi_0 = \Z$. By \autoref{transfer_on_pi0}, $q^!$ multiplies by $\chi(B)$, $p^!$ multiplies by $\chi(F)$, and $(q \circ p)^!$ multiplies by $\chi(E)$. So if the transfer is functorial on $\pi_0$, then the Euler characteristic is multiplicative.
\end{proof}

\section{First proof}

In this section we give a proof of \autoref{intro_multiplicative}. By \autoref{equivalence}, this also proves \autoref{intro_functorial}. We start by reducing to the case of an $n$-sheeted cover $\{1,\ldots,n\} \to E \to B$.

%

\begin{prop}
	The Euler characteristic is multiplicative for all fiber sequences if it is multiplicative for all fiber sequences with $F$ discrete (i.e. $n$-sheeted covers of finitely dominated spaces).
\end{prop}

\begin{proof}
Let $F \to E \to B$ be any fiber sequence with $B$, $E$, and $F$ finitely dominated. The monodromy of the fibration defines a homomorphism
$$
\pi_1(B) \to \text{GL}(H_\ast(F;\Bbb Z/2))
$$
into a finite group. Its kernel is a finite index subgroup, which determines a finite covering space $\tilde B \to B$.

The base change (pullback) of the fibration $E\to B$ to $\tilde B$ yields a new fibration
$\tilde E\to \tilde B$ with fiber $F$. This new fibration is orientable in the sense that the action of $\pi_1(\tilde B)$ on $H_*(F;\Z/2)$ is trivial. Therefore the $E^2$ page of its Serre spectral sequence is $H_p(\tilde B;H_q(F;\Z/2))$, where the coefficients are untwisted.

By \autoref{extension_of_fd} applied to $n$-sheeted covers, the spaces $\tilde E$ and $\tilde B$ are again finitely dominated, so they have Euler characteristics. As mentioned in the introduction, these Euler characteristics can be computed from the homology with mod 2 coefficients. It follows by a standard spectral sequence argument (see e.g. \cite[p.481]{spanier}) that $\chi(\tilde E) = \chi(\tilde B)\chi(F)$.

If we have multiplicativity for $n$-sheeted covers, then the two $n$-sheeted covers $\tilde E \to E$ and $\tilde B \to B$ give
\[ \chi(\tilde E) = n\chi(E), \quad \chi(\tilde B) = n\chi(B). \]
We conclude that
\[ \chi(E) = \frac1n \chi(\tilde E) = \frac1n \chi(\tilde B)\chi(F) = \chi(B)\chi(F). \qedhere \]
\end{proof}

We next recall some algebraic constructions. Let $\pi$ be any discrete group and let $\Z[\pi]$ be the corresponding group ring. Recall that $K_0(\Z[\pi])$ is the group completion of the set of isomorphism classes of finitely generated projective $\Z[\pi]$-modules $P$ under direct sum.

Equivalently, $K_0(\Z[\pi])$ is the free abelian group on the set of quasi-isomorphism classes of perfect chain complexes over $\Z[\pi]$, modulo the relation $[A]+[C] = [B]$ for any exact triangle $A\to B\to C$ in the derived category of perfect chain complexes. A chain complex is perfect if it is a retract up to quasi-isomorphism of a bounded chain complex of finitely generated free modules. Equivalently, if it is quasi-isomorphic to a bounded complex of finitely generated projective modules. See e.g. \cite{weibel_kbook} for more details.

We adopt the convention that all modules and chain complexes have \emph{right} action by $\Z[\pi]$.

The groups $K_0(\Z[\pi])$ define a functor in $\pi$, by extension of scalars. The homomorphisms $1 \to \pi \to 1$ make $K_0(\Z) \cong \Z$ into a retract of $K_0(\Z[\pi])$ corresponding to the free $\Z[\pi]$-modules. We let $\tilde K_0(\Z[\pi])$ be the complementary summand.

\begin{df}
	We define the \emph{rank} of a perfect $\Z$-chain complex $P$ to be its image in $K_0(\Z) \cong \Z$.
\end{df}
For bounded complexes of finitely generated projective modules, this is computed by taking the alternating sum of the ranks.

\begin{lem}\label{rank=euler}
	If $B$ is any finitely dominated space, the rank of the singular chain complex $C_*(B)$, as a chain complex of abelian groups, equals the Euler characteristic $\chi(B)$.
\end{lem}
\begin{proof}
	This can be proven in several ways -- for instance, suspending $B$ twice makes it finite without changing the rank of $C_*(B)$, and then it is equivalent to a bounded chain complex of finitely generated free modules, so its rank is equal to $\chi(B)$.
\end{proof}

We next recall a result of Swan from 
\cite{swan_finite}, and a corollary of this result suggested by Oscar Randall-Williams.

\begin{thm}[Swan]\label{swan}
	If $\pi$ is finite then $\tilde K_0(\Z[\pi])$ is torsion.
\end{thm}

Let $F$ be any finite set with left $\pi$-action. Let $F^t$ denote the same finite set with trivial $\pi$-action. Note that $\pi$ does not have to be finite.

\begin{prop}\label{same_rank}
For each perfect $\Z[\pi]$-chain complex $P$, the perfect $\Z$-chain complexes
$$
P \otimes_{\Z[\pi]} \Z[F] \quad \text{ and }\quad 
P \otimes_{\Z[\pi]} \Z[F^t] 
$$
have the same rank.
\end{prop}

\begin{proof}
Let $\Pi = \text{iso}(F)$, the automorphisms of
the set $F$. Then $\Pi$  is finite, and the $\pi$-action on $F$ defines homomorphism
 $\pi \to \Pi$. We get isomorphisms of chain complexes
\begin{align*}
	P \otimes_{\Z[\pi]} \Z[F] &\cong (P \otimes_{\Z[\pi]} \Z[\Pi]) \otimes_{\Z[\Pi]} \Z[F]\, , \\
	P \otimes_{\Z[\pi]} \Z[F^t] &\cong (P \otimes_{\Z[\pi]} \Z[\Pi]) \otimes_{\Z[\Pi]} \Z[F^t].
\end{align*}
It therefore suffices to prove the theorem for the finite group $\Pi$ and the perfect $\Z[\Pi]$-chain complex $P \otimes_{\Z[\pi]} \Z[\Pi]$.

Without loss of generality, therefore, $\pi$ is finite. We next observe that the two complexes $P\otimes_{\Bbb Z[\pi]} \Bbb Z[F]$ and $P\otimes_{\Bbb Z[\pi]} \Bbb Z[F^t]$ clearly have the same rank when $P$ is a bounded complex of free $\Z[
\pi]$-modules, so it suffices to extend this to $P$ perfect. To do this we write the operations as homomorphisms
$$
\rho,\rho^t: K_0(\Bbb Z[\pi]) \to K_0(\Bbb Z)
$$
defined by $\rho(P) = P\otimes_{\Bbb Z[\pi]} \Bbb Z[F]$, and 
$\rho^t(P) = P\otimes_{\Bbb Z[\pi]} \Bbb Z[F^t]$.

These homomorphisms are of the form $\Z \oplus A \to \Z$ with $A$ a torsion group, by \autoref{swan}. We have checked that they agree on $\Z$. Since $A$ is torsion, it follows that they agree on all of $\Z \oplus A$. Therefore $P\otimes_{\Bbb Z[\pi]} \Bbb Z[F]$ and $P\otimes_{\Bbb Z[\pi]} \Bbb Z[F^t]$ have the same rank for any perfect complex $P$.
\end{proof}

\begin{rmk}
We could also prove \autoref{same_rank} from a more elementary result of Swan: if $\pi$ is finite and $P$ is any projective $\Z[\pi]$-module, then the rationalization $P \otimes \Q$ is a free $\Q[\pi]$-module \cite[Theorem 4.2]{swan_finite}. It follows that the homomorphisms $\rho$ and $\rho^t$ in the proof of \autoref{same_rank} become equal after rationalization. But the map $K_0(\Z) \to K_0(\Q)$ is an isomorphism, so $\rho$ and $\rho^t$ are equal before rationalization as well.
\end{rmk}

Finally, suppose $F \to E \to B$ is a fiber sequence with $B$ finitely dominated and connected, and $F$ a discrete set of cardinality $n$. We want to show that $\chi(E) = n\chi(B)$. Without loss of generality, $E$ is connected as well.

Let $\pi = \pi_1(B)$, which acts on the fiber $F$ on the left and on the universal cover $\ti B$ on the right. Let $F^t$ be the same set as $F$ but with trivial $G$-action. Then we have identifications
\[ E = \ti B \times_\pi F, \qquad B \times F = \ti B \times_\pi F^t. \]
Taking chains, we have identifications
\begin{align}\label{id1}
	C_*(E) &\cong C_*(\ti B) \otimes_{\Z[\pi]} \Z[F], \\
	\label{id2}
	C_*(B \times F) &\cong C_*(\ti B) \otimes_{\Z[\pi]} \Z[F^t].
\end{align}

\begin{thm}\label{main_step}
	The finitely dominated spaces $E$ and $B \times F$ have the same Euler characteristic.
\end{thm}


\begin{proof}
The chain complex of $\Z[\pi]$-modules $C_*(\ti B)$ is perfect and therefore defines a class in $K_0(\Z[\pi]) $. By \autoref{same_rank}, the chain complexes \eqref{id1} and \eqref{id2} have the same rank. By \autoref{rank=euler}, these ranks are the Euler characteristics of $E$ and $B \times F$, respectively.
\end{proof}

Since $F$ is a finite set, this immediately implies that $\chi(E) = \chi(F)\chi(B)$, proving \autoref{intro_multiplicative}.

\section{Second proof}

In this section we give a second, independent proof of \autoref{intro_functorial}. By \autoref{equivalence}, this also proves \autoref{intro_multiplicative}.

This second proof is based on the results of \cite{CCRY}, more precisely on the following formula for the Becker-Gottlieb transfer of a composite. For a space $A$, let $LA = \maps(S^1,A)$ denote the free loop space. Denote the inclusion of constant loops and the evaluation at $1 \in S^1$ by
\[ A \xrightarrow{c} LA \xrightarrow{ev_1} A. \]

\begin{thm}[{\cite[Prop 6.7, Cor 6.10]{CCRY}}]\label{thm : CCRY}
Let $p:X\to Y$ and $q: Y\to Z$ be fibrations with finitely dominated fibers. Then $(q\circ p)^!$ is the composite

$$\Sigma^\infty_+Z\xrightarrow{c} \Sigma^\infty_+ LZ \xrightarrow{\tau} \Sigma^\infty_+ LY\xrightarrow{\chi} \Sigma^\infty_+X$$ 

where 
\begin{itemize}
    \item $\tau : \Sigma^\infty_+ LZ\to \Sigma^\infty_+ LY$ is the free loop space transfer, also known as the THH-transfer (\cite{LM}), and 
    \item $\chi$ is a twisted version of the Becker-Gottlieb transfer of $p$.
\end{itemize}

On the other hand, $p^!\circ q^!$ is the same composite, except with
\[ \Sigma^\infty_+ LY\xrightarrow{ev_1} \Sigma^\infty_+ Y\xrightarrow{c} \Sigma^\infty_+ LY \]
inserted before the final map in the composition.
\end{thm}
We will only need to recall from \cite{CCRY} how $\chi$ acts on $\pi_0$. Let $\gamma$ be a loop in $Y$, based at $y$, and let $X_y := p^{-1}(y)$ denote the fiber of $p$ at $y$. The loop $\gamma$ induces a monodromy self-equivalence $\gamma_* : X_y\to X_y$, well-defined up to homotopy. This map has a trace in the symmetric monoidal category of spectra, defined as in \eqref{pretransfer} as the composite of coevaluation, $\gamma_*$, and an evaluation map:
\begin{equation}\label{lefschetz}
	\mathbb S
	\to \Sigma^\infty_+ X_y \sma (\Sigma^\infty_+ X_y)^\vee
	\to \Sigma^\infty_+X_y \sma (\Sigma^\infty_+X_y)^\vee
	\to \mathbb S.
\end{equation}
On $\pi_0$, this returns an integer $L(\gamma)$, the Lefschetz number of the monodromy action of $\gamma$. The action of $\chi$ on $\gamma$ is very close to this:
\begin{prop}[{\cite[Prop 5.7]{CCRY}}]\label{trace_desc}
	$\chi(\gamma) \in \pi_0(\Sigma^\infty_+ X)$ is the trace of the map
	\[ (1,\gamma_*)\colon X_y \to X \times X_y \]
	in the symmetric monoidal category of spectra.
\end{prop}
In general, if we have a symmetric monoidal category with unit $I$, and a dualizable object $A$ with dual $A^*$, the trace of a map $f\colon A \to B \otimes A$ is defined to be the composite of coevaluation, $f$, and evaluation
\begin{equation*}
	I
	\to A \otimes A^*
	\to B \otimes A \otimes A^*
	\to B \otimes I
	\cong B,
\end{equation*}
see e.g. \cite{becker1976transfer,dold_puppe}.

So, the trace from \autoref{trace_desc} is the composite
\begin{equation}\label{trace_of_monodromy}
	\mathbb S
	\to \Sigma^\infty_+ X_y\sma (\Sigma^\infty_+ X_y)^\vee
	\to \Sigma^\infty_+ X\sma \Sigma^\infty_+X_y\sma (\Sigma^\infty_+X_y)^\vee
	\to \Sigma^\infty_+ X
\end{equation}
of coevaluation, the map $(1,\gamma_*)$, and evaluation.

\begin{rmk}
The map $\chi$ can also be described above $\pi_0$ in essentially the same way, by phrasing this description more coherently and thus defining a map of spaces $LY\to \Omega^\infty\Sigma^\infty_+ X$.
\end{rmk}

By \autoref{transfer_additive}, we can assume without loss of generality that $X$, $Y$, and $Z$ are connected, so that their suspension spectra all have $
\pi_0 = \Z$. Then the composite \eqref{trace_of_monodromy} on $\pi_0$ simplifies to the Lefschetz number \eqref{lefschetz}. Applying \autoref{thm : CCRY}, the transfer $(q\circ p)^!$ on $\pi_0$ takes the generator $1 \in \Z$ to an element in $H_0(LY)$, which is a weighted sum of free loops
\begin{equation}\label{sum}
	\sum_i a_i \gamma_i, \quad a_i \neq 0,
\end{equation}
and then to the corresponding weighted sum of Lefschetz numbers
\[ \sum_i a_i L(\gamma_i). \]
On the other hand, $p^!\circ q^!$ replaces all the loops $\gamma_i$ by constant loops, and then takes their monodromy, giving
\[ \sum_i a_i L(\id) = \sum_i a_i \chi(X_y). \]

To show these agree, it suffices to show that $L(\gamma_i) = L(\id)$ for each of the loops $\gamma_i$ that appear in the sum \eqref{sum}.


By \cite[7.12]{LM}, we have a commutative diagram
$$\xymatrix{ \Sigma^\infty_+Z \ar@/_2em/[dd]_c \ar[d]\\  A(Z)\ar[r]\ar[d] & A(Y) \ar[d] \\ 
 \Sigma^\infty_+ LZ\ar[r]^\tau & \Sigma^\infty_+ LY,}$$
where $A(-)$ denotes Waldhausen's $A$-theory (\cite{1126}), and the vertical maps are the Dennis trace. In particular, on $\pi_0$, the image of the composite
\[ \Sigma^\infty_+ Z\to \Sigma^\infty_+ LZ\to \Sigma^\infty_+ LY \]
is included in the image of $A(Y)\to \Sigma^\infty_+ LY$. This is well-known to agree on $\pi_0$ with the image of the Dennis trace map
\begin{equation}\label{dennis_trace}
	K_0(\mathbb Z[\pi_1(Y)])\to \HH_0(\mathbb Z[\pi_1(Y)]).
\end{equation}
To be specific, $A(Y)$ can be identified with the $K$-theory of the ring spectrum $\Sigma^\infty_+ \Omega Y$, and $\Sigma^\infty_+ LY$ is identified with its topological Hochschild homology ($\THH$). The map of ring spectra $\Sigma^\infty_+ \Omega Y \to H(\Z[\pi_1(Y)])$ is 1-connected, and so induces an isomorphism on $\pi_0$ for both $K$-theory and $\THH$. This gives a commuting diagram
\[ \xymatrix{
	\pi_0(A(Y)) \ar[d] \ar@{<-}[r]^-\cong & K_0(\Sigma^\infty_+ \Omega Y) \ar[d] \ar[r]^-\cong & K_0(\mathbb Z[\pi_1(Y)]) \ar[d] \ar[rd]^-{\eqref{dennis_trace}} \\
	\pi_0(\Sigma^\infty_+ LY) \ar@{<-}[r]^-\cong & \THH_0(\Sigma^\infty_+ \Omega Y) \ar[r]^-\cong & \THH_0(\mathbb Z[\pi_1(Y)]) \ar[r]^-\cong & \HH_0(\mathbb Z[\pi_1(Y)])
} \]
where the diagonal map and the vertical maps are all variants of the Dennis trace. See for example \cite[1.6 and 7.1]{EKMM}.

Therefore the loops $\gamma_i$ that appear in \eqref{sum} are those conjugacy classes in $\pi_1(Y)$ that are in the image of \eqref{dennis_trace}, in the sense that some element of $K_0$ has image whose coefficient in that conjugacy class is nonzero. The following theorem of Linnell characterizes which such loops $\gamma_i$ can appear.

\begin{thm}[\cite{linnell}, \cite{berrick_hesselholt}]\label{thm : Linnell}
For any group $G$ and any element $g \in G$, if some element of $K_0$ has image whose coefficient in the conjugacy class of $g$ is nonzero, then there exists an integer $m\geq 1$ such that for any integer $s\geq 1$, $g$ is conjugate to $g^{s^m}$. 
\end{thm}

In particular, if $g$ has this property and has finite order $n$, it must be trivial, since $g$ is conjugate to $g^{n^m} = 1$, and therefore $g = 1$.

\begin{cor}
If $\gamma_i$ appears in the sum \eqref{sum}, then $\gamma_i$ acts trivially on $H_*(X_y;k)$ for any finite field $k$. 
\end{cor}
\begin{proof}
Since $X_y$ is finitely dominated, $H_*(X_y;k)$ is a finite dimensional $k$-vector space. So if $k$ is a finite field, $GL(H_*(X_y;k))$ is finite. If $\gamma_i$ appears in the sum, then by \autoref{thm : Linnell} there is an $m$ such that $\gamma_i$ is conjugate to $\gamma_i^{s^m}$ for any $s \geq 1$. The same is true of the image of $\gamma_i$ in the finite group $GL(H_*(X_y;k))$. Therefore this image is trivial.
\end{proof}

Therefore each $\gamma_i$ in \eqref{sum} acts trivially on $H_*(X_y;\F_p)$ for each prime $p$. It follows that the Lefschetz number $L(\gamma_i)$ mod $p$, which can be computed from this action on $H_*(X_y;\F_p)$, agrees with $L(\id) = \chi(X_y)$. But this is true for every prime, so $L(\gamma_i) = \chi(X_y)$ integrally.

In summary, each $\gamma_i$ has the same Lefschetz number as the identity map of $X_y$. Therefore the image of $1 \in \Z$ under the two maps $(q\circ p)^!$ and $p^!\circ q^!$ agree,
\[ \sum_i a_i L(\gamma_i) = \sum_i a_i \chi(X_y). \]
This concludes the proof of \autoref{intro_functorial}.

\begin{rmk}
The result \cite[Thm 1.2]{LM} could be used in the place of \autoref{thm : CCRY} from \cite{CCRY} in this proof, if we also know that $\pi_1(Y)$ satisfies Bass' trace conjecture, because then the only loop $\gamma_i$ that can appear in \eqref{sum} is the trivial loop. See also \cite[Rmk 8.10]{LM}. 
\end{rmk}

\bibliographystyle{amsalpha}
\bibliography{references}{}

\end{document}

%% file: preamble_refs.tex
\usepackage{xcolor}
\definecolor{darkgreen}{rgb}{0,0.30,0} 
\definecolor{darkred}{rgb}{0.75,0,0}
\definecolor{darkblue}{rgb}{0,0,0.6} 
\usepackage[pdfborder=0,pagebackref,colorlinks,citecolor=darkgreen,linkcolor=darkgreen,urlcolor=darkblue]{hyperref}
\renewcommand*{\backref}[1]{}
\renewcommand*{\backrefalt}[4]{({%
    \ifcase #1 Not cited.%
          \or On p.~#2%
          \else On pp.~#2%
    \fi%
    })}

\def\makeautorefname#1#2{\expandafter\def\csname#1autorefname\endcsname{#2}}
\makeautorefname{equation}{Equation}%
\makeautorefname{footnote}{footnote}%
\makeautorefname{item}{item}%
\makeautorefname{figure}{Figure}%
\makeautorefname{table}{Table}%
\makeautorefname{part}{Part}%
\makeautorefname{appendix}{Appendix}%
\makeautorefname{chapter}{Chapter}%
\makeautorefname{section}{Section}%
\makeautorefname{subsection}{Section}%
\makeautorefname{subsubsection}{Section}%
\makeautorefname{paragraph}{Paragraph}%
\makeautorefname{subparagraph}{Paragraph}%
\makeautorefname{theorem}{Theorem}%
\makeautorefname{thm}{Theorem}%
\makeautorefname{bigthm}{Theorem}%
\makeautorefname{addm}{Addendum}%
\makeautorefname{mainthm}{Main theorem}%
\makeautorefname{corollary}{Corollary}%
\makeautorefname{cor}{Corollary}%
\makeautorefname{lemma}{Lemma}%
\makeautorefname{lem}{Lemma}%
\makeautorefname{sublemma}{Sublemma}%
\makeautorefname{sublem}{Sublemma}%
\makeautorefname{subl}{Sublemma}%
\makeautorefname{prop}{Proposition}%
\makeautorefname{property}{Property}
\makeautorefname{pro}{Property}
\makeautorefname{sch}{Scholium}%
\makeautorefname{step}{Step}%
\makeautorefname{conject}{Conjecture}%
\makeautorefname{conj}{Conjecture}%
\makeautorefname{questn}{Question}
\makeautorefname{quest}{Question}
\makeautorefname{qn}{Question}
\makeautorefname{definition}{Definition}%
\makeautorefname{defn}{Definition}%
\makeautorefname{defi}{Definition}%
\makeautorefname{def}{Definition}%
\makeautorefname{dfn}{Definition}%
\makeautorefname{df}{Definition}%
\makeautorefname{notation}{Notation}
\makeautorefname{notn}{Notation}
\makeautorefname{rem}{Remark}%
\makeautorefname{rems}{Remarks}%
\makeautorefname{rmk}{Remark}%
\makeautorefname{rk}{Remark}%
\makeautorefname{remarks}{Remarks}%
\makeautorefname{rems}{Remarks}%
\makeautorefname{rmks}{Remarks}%
\makeautorefname{rks}{Remarks}%
\makeautorefname{example}{Example}%
\makeautorefname{examp}{Example}%
\makeautorefname{exmp}{Example}%
\makeautorefname{exam}{Example}%
\makeautorefname{exa}{Example}%
\makeautorefname{ex}{Example}%
\makeautorefname{axiom}{Axiom}%
\makeautorefname{axi}{Axiom}%
\makeautorefname{ax}{Axiom}%
\makeautorefname{case}{Case}%
\makeautorefname{claim}{Claim}%
\makeautorefname{clm}{Claim}%
\makeautorefname{assumpt}{Assumption}%
\makeautorefname{asses}{Assumptions}%
\makeautorefname{conclusion}{Conclusion}%
\makeautorefname{concl}{Conclusion}%
\makeautorefname{conc}{Conclusion}%
\makeautorefname{cond}{Condition}%
\makeautorefname{const}{Construction}%
\makeautorefname{con}{Construction}%
\makeautorefname{criterion}{Criterion}%
\makeautorefname{criter}{Criterion}%
\makeautorefname{crit}{Criterion}%
\makeautorefname{exercise}{Exercise}%
\makeautorefname{exer}{Exercise}%
\makeautorefname{exe}{Exercise}%
\makeautorefname{warn}{Warning}%
\newtheorem{thm}{Theorem}[section]
\newtheorem{cor}{Corollary}[section]
\newtheorem{lem}{Lemma}[section]
\newtheorem{prop}{Proposition}[section]
\theoremstyle{definition}
\newtheorem{df}{Definition}[section]

\newtheorem{rmk}{Remark}[section]

\numberwithin{equation}{section}
\numberwithin{figure}{section}

\newtheorem{bigthm}{Theorem}

\theoremstyle{definition}

\makeatletter
\let\c@cor=\c@thm
\let\c@prop=\c@thm
\let\c@lem=\c@thm
\let\c@df=\c@thm
\let\c@ex=\c@thm
\let\c@warn=\c@thm
\let\c@rmk=\c@thm
\let\c@equation\c@thm
\let\c@figure\c@thm
\let\c@table\c@thm
\makeatother

\hypersetup{linktoc=all}